\documentclass{amsart}
\usepackage{amsmath}
\usepackage{amsthm}
\usepackage{amssymb}
\usepackage{bm}
\usepackage{dsfont}
\usepackage{enumitem}\setlist[enumerate]{nosep,label=\textnormal{(\arabic*)}}
\usepackage{eucal}
\usepackage{graphicx}
\usepackage[hidelinks]{hyperref}
\usepackage[capitalise]{cleveref}
\usepackage{mathrsfs}
\usepackage{mathtools}
\usepackage{nameref}
\usepackage{physics}
\usepackage{seqsplit}
\usepackage{stmaryrd}
\usepackage{thmtools}
\usepackage{xcolor}
\usepackage{xparse}

\usepackage{tikz}
\usetikzlibrary{cd}
\usetikzlibrary{positioning, angles, quotes}

\theoremstyle{plain}
\newtheorem{thm}{Theorem}[section] \newtheorem*{thm*}{Theorem}
\newtheorem{cor}[thm]{Corollary} \newtheorem*{cor*}{Corollary}
 
 \newtheorem*{prop*}{Proposition}
\newtheorem{lem}[thm]{Lemma} \newtheorem*{lem*}{Lemma}
\newtheorem{q}[thm]{Question} \newtheorem*{q*}{Question}
 \newtheorem*{conj*}{Conjecture}
\newtheorem{claim}[thm]{Claim}

\theoremstyle{definition}
\newtheorem{defn}[thm]{Definition} \newtheorem*{defn*}{Definition}

\theoremstyle{remark}
\newtheorem{rem}[thm]{Remark} \newtheorem*{rem*}{Remark}
 \newtheorem*{ex*}{Example}

\theoremstyle{plain}

\Crefname{defn}{Definition}{Definitions}

\newcommand{\FP}{\mathrm{FP}}

\renewcommand{\b}[1]{b^{(2)}_{#1}}

\newcommand{\Dk}[1]{\mathcal{D}_{k[{#1}]}}

\newcommand{\N}{\mathbb{N}}
\newcommand{\Q}{\mathbb{Q}}

\newcommand{\Z}{\mathbb{Z}}

\newcommand{\inv}{^{-1}}

\DeclareMathOperator{\cd}{cd}

\DeclareMathOperator{\h}{H}
\DeclareMathOperator{\Hom}{Hom}

\DeclareMathOperator{\supp}{supp}

\newcounter{pablocomments}

\newcounter{samcomments}

\title{On the top-dimensional $L^2$-Betti number of residually poly-$\Z$ groups}

\author{Sam P.\  Fisher}
\author{Pablo S\'anchez-Peralta}

\address{Instituto de Ciencias Matem\'aticas, CSIC-UAM-UC3M-UCM}
\email{samuel.p.fisher@gmail.com}

\address{Departamento de Matem\'aticas, Universidad Aut\'onoma de Madrid}
\email{pablo.sanchezperalta@uam.es}

\begin{document}

\begin{abstract}
    Let $G$ be a residually poly-$\Z$ group of finite type. We prove that $G$ admits a poly-$\Z$ quotient with kernel $N$ satisfying $\cd_\Q(N) < \cd_\Q(G)$ if and only if the top-dimensional $L^2$-Betti number of $G$ vanishes.
\end{abstract}

\maketitle

\section{Introduction}

The existence of an infinite-index normal subgroup satisfying relatively mild finiteness conditions in a group $G$ has strong implications on the vanishing of the $L^2$-Betti numbers $\b{i}(G)$. This is perhaps best exemplified by Gaboriau's theorem \cite[Th\'eor\`eme 6.8]{Gaboriau2002}, which states that if $G$ has a finitely generated infinite normal subgroup of infinite index, then $\b{1}(G) = 0$. If one assumes there is a normal subgroup $N$ such that $G/N$ is infinite and amenable, then for any $n$ we have that $\b{n}(N) < \infty$ implies $\b{n}(G) = 0$ \cite[Th\'eor\`eme 6.6]{Gaboriau2002}. This has the following consequences: if $N \trianglelefteqslant G$ is a normal subgroup such that $G/N$ is infinite amenable, then

\begin{enumerate}
    \item $\b{i}(G) = 0$ for all $i \leqslant n$ if $N$ is of type $\FP_n(\Q)$;
    \item $\b{i}(G) = 0$ for all $i > \cd_\Q(N)$.
\end{enumerate}

On the other hand, it is difficult to deduce the existence of a normal subgroup with desirable finiteness properties from the knowledge that $\b{n}(G) = 0$ for some $n$. One of the most general classes of groups where this type of result has been obtained is in the class of RFRS groups, which can be defined as the class of residually \{poly-$\Z$ virtually Abelian\} groups (see \cite[Theorem 6.3]{OkunSchreve_DawidSimplified}). Kielak showed that if $G$ is a non-trivial finitely generated RFRS group, then $\b{1}(G) = 0$ if and only if $G$ has a finite-index subgroup mapping onto $\Z$ with finitely generated kernel \cite{KielakRFRS}. In a different direction, the first author proved that if a non-trivial finitely generated RFRS group $G$ with $\cd(G) \leqslant 2$ has $\b{2}(G) = 0$, then $G$ is virtually free-by-cyclic \cite{Fisher_freebyZ}.

An attempt to generalise these results beyond the class of RFRS groups was made in \cite{FisherKlinge_RPVN}, where the class of RPVN (for residually \{poly-$\Z$ virtually nilpotent\}) groups was studied. There, the initial goal was to prove that Kielak's fibring theorem holds for RPVN groups. While this was not achieved, an analogue of \cite[Theorem A]{Fisher_freebyZ} was obtained: if $G$ is a finitely generated RPVN group of cohomological dimension at most two, then $G$ is free-by-\{poly-$\Z$ virtually nilpotent\} if and only if $\b{2}(G) = 0$ \cite[Corollary C]{FisherKlinge_RPVN}. The purpose of this note is to show that the virtual nilpotence hypothesis is unnecessary, as well as to give a new and simpler proof.

\begin{thm}\label{thm:intro_main}
    Let $G$ be a residually poly-$\Z$ group of type $\FP(\Q)$ and of finite rational cohomological dimension $\cd_\Q(G) = n$. The following are equivalent:
    \begin{enumerate}
        \item $\b{n}(G) = 0$;
        \item there is a normal subgroup $N \trianglelefteqslant G$ such that $\cd_\Q(N) < n$ and $G/N$ is poly-$\Z$.
    \end{enumerate}
\end{thm}

\cref{thm:intro_main} will follow from the more precise and general \cref{thm:main}, which also holds with mod $p$ analogues of $L^2$-Betti numbers. In dimension two, an application of the Stallings--Swan theorem yields the following corollary.

\begin{cor}\label{cor:free_by_poly_Z}
    Let $G$ be a finitely generated residually poly-$\Z$ group with $\cd_\Q(G)$ at most $2$. Then $G$ is free-by-\{poly-$\Z$\} if and only if $\b{2}(G) = 0$.
\end{cor}

The results of \cite{KielakRFRS,Fisher_freebyZ,FisherKlinge_RPVN} all rely on a close connection between the $L^2$-homology of a group and its homology with coefficients in certain Novikov rings. The broad outline of each proof is to use vanishing of $L^2$-(co)homology to deduce vanishing of Novikov (co)homology for some finite-index subgroup, and then use vanishing Novikov (co)homology to understand the kernel of the relevant map to $\Z$, or more generally to a nilpotent group. This method is not as directly applicable in the context of residually poly-$\Z$ groups since a good analogue of the Novikov ring for poly-$\Z$ groups does not seem to exist; this is essentially due to the fact that poly-$\Z$ groups are not bi-orderable in general. Instead we show that $\b{n}(G) = 0$ implies that $\h^n(G;\mathcal R) = 0$ for $n = \cd(G)$, where $\mathcal R$ is an \emph{inductive ring} as defined by Okun and Schreve in \cite{OkunSchreve_DawidSimplified} (see \cref{sec:okun_schreve}). The vanishing of $\h^n(G;\mathcal R)$ is sufficient to obtain a dimension drop result.

It is interesting to speculate about the most general class of group where an analogue of \cref{cor:free_by_poly_Z} holds.

\begin{q}\label{q}
    Let $\mathcal C$ be a product closed class of locally indicable groups, and let $G$ be a finitely generated residually $\mathcal C$ group with $\cd_\Q(G) \leqslant 2$. If $\b{2}(G) = 0$, then is $G$ free-by-$\mathcal C$?
\end{q}

Linton proved that \cref{q} has a positive answer if $G$ is a one-relator group \cite[Corollary 3.3]{Linton_residuallyQsolvable}. It also has a positive answer for limit groups by a result of Kochloukova \cite[Theorem A]{Kochloukova_subdirect}, and here $\mathcal C$ can be any class of torsion-free groups.

\subsection*{Acknowledgments}

The first author is supported by  the grant \seqsplit{CEX2023-001347-S} of the Ministry of Science, Innovation and Universities of Spain. The second author is supported by the grants \seqsplit{PID2020-114032GB-I00} and \seqsplit{PID2024-155800NB-C33} of the Ministry of Science, Innovation and Universities of Spain.

\section{Preliminaries}

Throughout this section, $k$ always denotes a field. A group is \emph{locally indicable} if each of its non-trivial finitely generated subgroups admits a map onto $\Z$. A group $G$ is \emph{poly-$\Z$} if there is a subnormal series
\[
    \{1\} = G_n \leqslant G_{n-1} \leqslant \dots \leqslant G_1 \leqslant G_0  = G
\]
for some $n \geqslant 0$ such that $G_{i}/G_{i+1} \cong \Z$ for each $0 \leqslant i < n$. A group is \emph{residually poly-$\Z$} if it admits a chain
\[
    G = N_0 \geqslant N_1 \geqslant \dots
\]
such that $\bigcap_{i \in \N} N_i = \{1\}$, each subgroup $N_i$ is normal in $G$, and $G/N_i$ is poly-$\Z$ for each $i$. Since a group extension of locally indicable groups is locally indicable, it follows by induction that poly-$\Z$ groups are locally indicable. Moreover, residually locally indicable groups are locally indicable, so residually poly-$\Z$ groups are locally indicable.

\begin{rem}
    An arguably more standard definition would be to say that $G$ is residually poly-$\Z$ if every non-trivial element of $G$ survives in a poly-$\Z$ quotient. Since the class of poly-$\Z$ groups is closed under products, this definition is equivalent to the one given above for countable groups. We chose to give the definition above since it is what we use in the proofs, and because in any case the groups we deal with are finitely generated.
\end{rem}

Let $R$ be a ring. An \emph{$R$-ring} is a pair $(S, \varphi)$ where $\varphi \colon R \rightarrow S$ is a ring homomorphism. We will often omit $\varphi$ if it is clear from the context. Two $R$-rings $(S_1, \varphi_1)$ and $(S_2, \varphi_2)$ are \emph{isomorphic} if there exists a ring isomorphism $\alpha \colon S_1 \to S_2$ such that $\varphi_2 = \alpha \circ \varphi_1$. A ring homomorphism $\varphi$ from $R$ to a division ring $\mathcal{D}$ is \emph{epic} if the division closure of the image of $R$ equals $\mathcal{D}$.

\begin{defn}
    Let $G$ be a locally indicable group and let $\mathcal D$ be a division ring such that there is an epic embedding $\varphi \colon k[G] \hookrightarrow \mathcal D$. The embedding $\varphi$ is called \emph{Hughes-free} if for every non-trivial finitely generated subgroup $H \leqslant G$ and every $N \triangleleft H$ such that $H/N \cong \Z$, the multiplication map 
    \[
        \mathcal D_N \otimes_{k[N]} k[H] \rightarrow \mathcal D, \qquad x \otimes y \mapsto x \cdot \varphi(y)
    \]
    is injective. Here, $\mathcal D_N$ denotes the division closure of $\varphi(k[N])$ in $\mathcal D$. In this case, we call $\mathcal D$ a \emph{Hughes-free division ring} for $k[G]$.
\end{defn}

Hughes \cite{HughesDivRings1970} proved that if $k[G]$ admits a Hughes-free embedding, then it is unique up to $k[G]$-algebra isomorphism. In this situation, we thus speak of \emph{the} Hughes-free division ring of $k[G]$, denote it by $\Dk{G}$, and view $k[G]$ as a subring of $\Dk{G}$. If $H$ is any subgroup of $G$, then $\Dk{H}$ is isomorphic to the division closure of $k[H]$ in $\Dk{G}$, so we view $\Dk{H}$ as a subring of $\Dk{G}$.

The following result of Jaikin-Zapirain shows that Hughes-free division rings always exist for group algebras of residually poly-$\Z$ groups.

\begin{thm}[{\cite[Corollary 1.3]{JaikinZapirain2020THEUO}}]
    If $G$ is a residually \{locally indicable and amenable\} group, then $\Dk{G}$ exists for any field $k$.
\end{thm}

If $\mathcal D$ is a division $k[G]$-ring, then $\mathcal D$ inherits the structure of a $k[G]$-module, and therefore we can compute homology of $G$ with coefficients in $\mathcal D$. Recall that every module over $\mathcal D$ has a well-defined dimension.

\begin{defn}
    Let $\mathcal D$ be a division $k[G]$-ring and $n$ a non-negative integer. The $n$th \emph{$\mathcal D$-Betti number} of $G$ is the (possibly infinite) value
    \[
        b_n^\mathcal D (G) = \dim_\mathcal D \h_n(G;\mathcal D) \in \Z_{\geq 0} \cup \{\infty\}.
    \]
\end{defn}

It is not difficult to see that $\h_n(G;\mathcal D)$ and $\h^n(G;\mathcal D)$ are either both infinite-dimensional (but perhaps of different cardinalities) or have the same (finite) dimension (the proof of this is the same as the proof that $\h_n(G;\Q)$ and $\h^n(G;\Q)$ are isomorphic as $\Q$-vector spaces when they are both finite-dimensional). This observation will be important in the proof of \cref{thm:main}, since we will work with the vanishing of cohomology rather than homology.

When $G$ is locally indicable, $\mathcal D_{\Q[G]}$ always exists and coincides with the division closure of $\Q[G]$ in $\mathcal U(G)$, the algebra of affiliated operators, by a result of Jaikin-Zapirain and L\'opez-\'Alvarez \cite{JaikinLopezStrongAtiyah2020}. It follows that $b_n^{\mathcal D_{\Q[G]}}(G)$ coincides with the classical $n$th $L^2$-Betti number of $G$, denoted by $\b{n}(G)$. We refer the reader to \cite[Chapter 10]{Luck02} for more information on this topic. When $\mathcal D_{\mathbb F_p[G]}$ exists, the invariants $b_n^{\mathcal D_{\mathbb F_p[G]}}(G)$ provide a natural mod $p$ generalisation of the $L^2$-Betti numbers, and behave in a very similar way. For this reason they are sometimes denoted by $\b{n}(G;\mathbb F_p)$.

We will also need the following construction of division rings, due to Malcev \cite{Malcev_series} and Neumann \cite{Neumann_series}. First, we say that a ring $S$ is a \emph{crossed product} of a ring $R$ and a group $G$ if its underlying additive group decomposes as a direct sum $\bigoplus_{g \in G} R_g$, where
\begin{enumerate}
    \item each $R_g$ is an additive subgroup of $S$;
    \item $R_g R_h \subseteq R_{gh}$ for all $g,h \in G$;
    \item there are units $u_g \in R_g$ for all $g \in G$;
    \item $R_e$ is a subring of $S$ isomorphic to $R$.
\end{enumerate}
In this case, we denote $S$ by $R * G$, even though the crossed product structure may not be unique. By an abuse of notation, we will usually denote the unit $u_g$ by $g$; with this notation, every element of $R*G$ is expressible as a finite formal sum $\sum_{g \in G} r_g g$ with $r_g \in R$ for each $g \in G$. Given a (possibly infinite) formal sum $r = \sum_{g \in G} r_g g$ its \emph{support}, denoted $\supp(r)$, is defined as the set $\{g \in G : r_g \neq 0\}$.

\begin{defn}
    Let $G$ be a bi-orderable group, let $R$ be a ring, and let $R * G$ be a crossed product. The \emph{Malcev--Neumann power series} ring $R *_< G$ has underlying set all formal power series 
    \[
        R *_< G = \left\{ \sum_{g \in G} r_g g \ : \ r_g \in R, \  \supp(r) \ \text{is well-ordered w.r.t.} < \right\}.
    \]
    The addition on $R *_< G$ is defined pointwise, and the multiplication naturally extends that on $R*G$ (see \cite{Malcev_series} or \cite{Neumann_series} for details).
\end{defn}

The main theorem of Malcev and Neumann is that $R*_< G$ is a division ring whenever $R$ is a division ring.

Suppose that $G$ is a locally indicable group such that $\Dk{G}$ exists and let $H$ be a normal subgroup of $G$ such that $G/H$ is bi-orderable; fix an order $<$ on $G/H$. By the uniqueness of Hughes-free division rings, the conjugation action of $G$ on $k[H]$ extends to an action on $\Dk{H}$, and it is thus possible to form the crossed product $\Dk{H}*G/H$.

For any subgroup $A \leqslant G$, there is an embedding of Malcev--Neumann rings
\[
    \mathcal K_A := \Dk{A \cap H} *_< A/A \cap H \hookrightarrow \Dk{H} *_< G/H,
\]
and we leave it as an exercise to verify that the multiplication map
\[
    \mathcal K_A \otimes_{k[A]} k[G] \rightarrow \Dk{H} *_< G/H
\]
is injective. This implies that the division closure of $k[G]$ in $\Dk{H} *_< G/H$ is a Hughes-free division ring for $k[G]$, and therefore, by Hughes's uniqueness theorem \cite{HughesDivRings1970}, there is an embedding $\iota_< \colon \Dk{G} \hookrightarrow \Dk{H} *_< G/H$.

\section{The Okun--Schreve construction}\label{sec:okun_schreve}

In this section we recount the construction of inductive rings due to Okun and Schreve \cite{OkunSchreve_DawidSimplified}, and construct a class of Malcev--Neumann type modules over the inductive rings which will be used in the proof of the main result. We continue to use $k$ to denote an arbitrary (but fixed) field.

Let $G$ be a locally indicable group such that $\Dk G$ exists, and let $B \leqslant A$ be a pair of normal subgroups of $G$ such that $A/B$ is bi-orderable and amenable. Then each choice of bi-ordering $<$ on $A/B$ induces an injective homomorphism
\[
    \iota_< \colon \Dk{A} \hookrightarrow \Dk{B} *_< A/B
\]
as explained at the end of the previous section.

Let $R \leqslant \Dk{B}$ be a subring that is invariant under the conjugation action of $G$. It then makes sense to consider the subring $R *_< A/B$ of $\Dk{B} *_< A/B$, i.e.\ the subring of all formal power series with coefficients in $R$ and well-ordered support.

\begin{defn}
     With the above notation, the corresponding \emph{invariant Malcev--Neumann ring} is the subring $R * \langle A/B \rangle$ of $\Dk{A}$ consisting of the elements whose images under $\iota_<$ lie in $R *_< A/B$ for each bi-ordering $<$ on $A/B$. More formally,
    \[
        R * \langle A/B \rangle = \bigcap_{<} \iota_<\inv(R *_< A/B),
    \]
    where the intersection is taken over all bi-orderings of $A/B$.
\end{defn}

Importantly, $R * \langle A/B \rangle$ contains $R * A/B$ and it is invariant under the conjugation action of $G$ on $\Dk{A}$.

\begin{rem}
    It is perhaps not clear why the amenability assumption was imposed on $A/B$. Indeed, it is not strictly needed in order to define the invariant Malcev--Neumann rings, but it is needed in Okun and Schreve's proof of \cref{thm:OS} below. The key point is that crossed products of division rings with locally indicable amenable groups are \emph{Ore domains}, and thus admit embeddings into division rings of formal fractions in much the same way that an integral domain admits an embedding into its field of formal fractions.
\end{rem}

Now suppose that $G$ (still a locally indicable group such that $\Dk{G}$ exists) admits a normal residual chain
\[
    G = N_0 \geqslant N_1 \geqslant N_2 \geqslant \dots
\]
such that $N_i/N_{i+1}$ is bi-orderable and amenable for all $i \geqslant 0$. For each $i \geqslant 0$ we may form the \emph{inductive ring}
\[
    \mathcal R_i = k[N_i] * \langle N_{i-1}/N_i \rangle * \cdots * \langle N_0/N_1 \rangle.
\]
We have chosen not to bracket the expression for $\mathcal R_i$ for notational ease, and because there is only one possible way to parse the expression for it to make sense. Note that $\mathcal R_i$ depends on the chosen residual chain for $G$, but this choice will be clear from the context, so we suppress it from the notation. Note that $\mathcal R_0 = k[G]$ and $\mathcal R_i \leqslant \mathcal R_{i+1}$ for all $i \geqslant 0$. We now state the main result of Okun and Schreve.

\begin{thm}[{\cite[Theorem 5.1]{OkunSchreve_DawidSimplified}}]\label{thm:OS}
    With the notation above, $\Dk{G} = \bigcup_{i \geqslant 0} \mathcal R_i$.
\end{thm}

We conclude this section with a new construction of modules over the inductive rings. Continuing with the notation above, we will say that a right order on $G/N_i$ is \emph{induced} (with respect to the chain $(N_i)_{i \geqslant 0}$) if it is obtained by choosing bi-orders on each of the quotients $N_j/N_{j+1}$ for $0 \leqslant j < i$, and then endowing $G/N_i$ with the induced lexicographic right-invariant order.

Let $M$ be a right $k[N_i]$-module, let $<$ be any right-invariant order on $G/N_i$, and let $T$ be a transversal for $N_i$ in $G$. Given a (possibly infinite) formal sum $m = \sum_{t \in T} m_t t$ define its \emph{support}, denoted $\supp(m)$, as the set $\{\overline t \in G/N_i : m_t \neq 0\}$. Define a right $k[G]$-module with underlying set
\[
    M *_< G/N_i = \left\{ \sum_{t \in T} m_t t \ : \ m_g \in M, \ \supp(m) \ \text{is well-ordered w.r.t.} < \right\}.
\]
The module structure is obtained by viewing $M *_< G/N_i$ as a subset of the coinduced module 
\[
    \operatorname{coInd}_{N_i}^G M = \Hom_{k[N_i]}(k[G],M).
\]
If $<$ was a bi-ordering, then $M *_< G/N_i$ would inherit a $k[N_i]*_< G/N_i$-module structure. In the case that $<$ is only induced, we still have the following key result.

\begin{lem}\label{lem:induced_order_module_structure}
    If $<$ is an induced right order on $G/N_i$ and $M$ is a right $k[N_i]$-module, then $M *_< G/N_i$ carries a right $\mathcal R_i$-module structure extending its $k[G]$-module structure.
\end{lem}
\begin{proof}
    For each $j < i$, let $<_j$ be a bi-order on $N_j/N_{j+1}$, and suppose that $<$ is induced from these orders. Then
    \[
        M *_< G/N_i \cong M *_{<_{i-1}} N_{i-1}/N_i *_{<_{i-2}} \cdots *_{<_0} N_0/N_1.
    \]
    By reverse induction on $j$, we will prove that 
    \[
        L_j = M *_< N_j/N_i \cong M *_{<_{i-1}} N_{i-1}/N_i *_{<_{i-2}} \cdots *_{<_j} N_j/N_{j+1}
    \]
    is an $S_j$-module, where 
    \[
        S_j = k[N_i] * \langle N_{i-1}/N_i \rangle * \cdots * \langle N_j/N_{j+1} \rangle.
    \]
    For $j = i$ there is nothing to prove; this is just the statement that $M$ is a $k[N_i]$-module. Now assume $j < i$. By induction,
    \[
        L_j \cong L_{j+1} *_{<_j} N_j/N_{j+1}
    \]
    is an $S_{j+1} *_{<_j} N_j/N_{j+1}$-module. But by definition, $\iota_{<_j}$ induces an injection
    \[
        S_j \cong S_{j+1} * \langle N_j/N_{j+1} \rangle \hookrightarrow S_{j+1} *_{<_j} N_j/N_{j+1},
    \]
    thus giving $L_j$ an $S_j$-module structure. Taking $j = 0$ concludes the proof. \qedhere
\end{proof}

\section{Proof of the main result}

We make a preliminary observation about residually poly-$\Z$ groups, which will allow us to use \cref{thm:OS}.

\begin{lem}\label{lem:res_polyZ_res_chain}
    Let $G$ be a residually poly-$\Z$ group. Any normal residual chain $(M_i)_{i \geqslant 0}$ of $G$ such that $G/M_i$ is poly-$\Z$ for all $i$ can be refined to a normal residual chain
    \[
        G = N_0 \geqslant N_1 \geqslant N_2 \geqslant \dots
    \]
    such that $N_i/N_{i+1}$ is a finitely generated free Abelian group for each $i \geqslant 0$.
\end{lem}
\begin{proof}
    Let $P$ be a poly-$\Z$ group. Note that $P$ is finitely generated. Then $P$ has a characteristic series
    \[
        \{1\} = P_n \leqslant \dots \leqslant P_1 \leqslant P_0 = P
    \]
    for some integer $n$, where each quotient is finitely generated free Abelian. Indeed, one can let $P_{i+1} = [P_i,P_i]$. Poly-$\Z$ groups---essentially by definition---have infinite Abelianisation, so $\cd(P_{i+1}) < \cd(P_i) \leqslant \cd(P) < \infty$. Since $P_{i+1}$ is again a poly-$\Z$ group, this process eventually terminates.

    We can now prove the lemma. Each consecutive quotient $M_i/M_{i+1}$ is poly-$\Z$, so by the paragraph above has a characteristic series with finitely generated free Abelian consecutive quotients. For each $i$, lift the terms of this characteristic series to $G$ to obtain a refinement of $(M_i)_{i \geqslant 0}$; this is the desired normal residual chain. \qedhere
\end{proof}

Recall that a group $G$ is of \emph{type $\FP(k)$} if the trivial $k[G]$-module $k$ admits a finite length resolution by finitely generated projective modules. For a commutative ring $R$, we say that $\cd_R(G) = n$ if $n$ is the largest integer such that there exists an $R[G]$-module $M$ with $\h^n(G;M) \neq 0$. \cref{thm:intro_main} follows immediately from the result below by taking $k = \Q$.

\begin{thm}\label{thm:main}
    Let $G$ be a residually poly-$\Z$ group of type $\FP(k)$, where $k$ is a field. Assume that $\cd_k(G) = n$. The following are equivalent:
    \begin{enumerate}
        \item\label{item:top_betti_zero} $b_n^{\mathcal D_{k[G]}}(G) = 0$;
        \item\label{item:dim_drop} for any normal residual chain $G = N_0 \geqslant N_1 \geqslant \dots$ such that $G/N_i$ is poly-$\Z$ for each index $i$, there is some $i \in \N$ such that $\cd_k(N_i) < n$.
    \end{enumerate}
\end{thm}
\begin{proof}
    That \ref{item:dim_drop} implies \ref{item:top_betti_zero} is a consequence of \cite[Proposition 3.18]{FisherKlinge_RPVN}.

    Now suppose that \ref{item:top_betti_zero} holds. We may assume that the normal residual chain $(N_i)_{i \geqslant 0}$ has finitely generated free Abelian consecutive quotients by \cref{lem:res_polyZ_res_chain}. 

    \begin{claim}\label{claim:pass_to_invMN}
        There is some $i$ such that $\h^n(G;\mathcal R_i) = 0$.
    \end{claim}
    \begin{proof}
        By \cite[Proposition VIII.6.8]{BrownGroupCohomology}, we have
        \[
            0 = \h^n(G;\Dk{G}) \cong \Dk{G} \otimes_{k[G]} \h^n(G;k[G]).
        \]
        Since $G$ is of type $\FP(k)$, the left $k[G]$-module $\h^n(G;k[G])$ is finitely generated, say by elements $x_1, \dots, x_s$ for some $s \in \N$.
    
        By \cref{thm:OS}, $\Dk{G} = \bigcup_{i \geqslant 0} \mathcal R_i$, where the rings $\mathcal R_i$ are the inductive rings associated to the residual normal chain $(N_i)$ and defined in \cref{sec:okun_schreve}. Because tensor products commute with directed limits, we obtain
        \[
            0 = \h^n(G;\Dk{G}) \cong \varinjlim_i \mathcal R_i \otimes_{k[G]} \h^n(G;k[G]).
        \]
        Hence, there is some $i$ such that $1 \otimes x_j = 0$ in $\mathcal R_i \otimes_{k[G]} \h^n(G;k[G])$ for all $j \leqslant s$, so we have 
        \[
            \h^n(G;\mathcal R_i) \cong \mathcal R_i \otimes_{k[G]} \h^n(G;k[G]) = 0,
        \]
        where we have again applied \cite[Proposition VIII.6.8]{BrownGroupCohomology}. \renewcommand\qedsymbol{$\diamond$}
    \end{proof}

    We will now show that $\cd_k(N_i) < n$. To this end, let $M$ be an arbitrary right $k[N_i]$-module; our goal is to prove that 
    \[
        \h^n(N_i; M) \cong \h^n(G; \operatorname{coInd}_{N_i}^G M) = 0.
    \]
    If $\operatorname{coInd}_{N_i}^G M$ were an $\mathcal R_i$-module, we would be done. While this is not the case in general, the following claim will be sufficient to obtain the vanishing.
    
    \begin{claim}\label{claim:epimorphism}
        There is a collection of $\mathcal R_i$-modules $L_1, \dots, L_{2^l}$ for some integer $l$ such that there exists an epimorphism $\bigoplus_{1 \leqslant p \leqslant 2^l} L_p \rightarrow \operatorname{coInd}_{N_i}^G M$ of $k[G]$-modules.
    \end{claim}
    \begin{proof}
        The poly-$\Z$ group $G/N_i$ has the characteristic series with finitely generated free Abelian consecutive quotients 
        \[
            \{1\} = N_i/N_i \leqslant N_{i-1}/N_i \leqslant \cdots \leqslant G/N_i.
        \]
        Fixing a basis for each consecutive quotient $N_j/N_{j+1}$ induces $2^{b_j}$ preferred lexicographic bi-orders on the free Abelian group, where $b_j$ is the rank of $N_j/N_{j+1}$. The modules $L_p$ in the statement are then the $2^l$ modules $M *_< G/N_j$ for each of the induced orders on $G/N_j$, where $l = b_0 + \cdots + b_{i-1}$. By \cref{lem:induced_order_module_structure}, each module $L_p$ is an $\mathcal R_i$-module.

        Each module $L_p$ is naturally a submodule of $\operatorname{coInd}_{N_i}^G M$, so the inclusion maps induce a homomorphism
        \[
            \bigoplus_{1 \leqslant p \leqslant 2^l} L_p \rightarrow \operatorname{coInd}_{N_i}^G M
        \]
        of $k[G]$-modules. Induction on $i$ and the following observation together show that this map is in fact a surjection: if $S \subseteq \Z^d$ is any subset, then there is a decomposition
        \[
            S = S_1 \cup \cdots \cup S_{2^d}
        \]
        where each $S_i$ is well-ordered with respect to at least one of the $2^d$ lexicographic bi-orders on $\Z^n$ coming from the standard basis elements (i.e.\ one could take $S_j$ to be the intersection of $S$ with the $j$th orthant). \renewcommand\qedsymbol{$\diamond$}
    \end{proof}

    We are now ready to conclude the proof. By \cref{claim:epimorphism}, the long exact sequence in group cohomology associated to a short exact sequence of coefficients, and the fact that $\cd_k(G) = n$, there is a surjection
    \[
        \bigoplus_{1 \leqslant p \leqslant 2^l} \h^n(G; L_p) \cong \h^n \left(G; \bigoplus_{1 \leqslant p \leqslant 2^l} L_p \right) \rightarrow \h^n(G; \operatorname{coInd}_{N_i}^G M).
    \]
    But \cref{claim:pass_to_invMN} and a final application of \cite[Proposition VIII.6.8]{BrownGroupCohomology} yield
    \[
        \h^n(G; L_p) \cong L_p \otimes_{\mathcal R_i} \h^n(G; \mathcal R_i) = 0,
    \]
    concluding the proof. \qedhere
\end{proof}

\begin{rem}
    The finiteness property $\FP(k)$ was used to apply \cite[Proposition VIII]{BrownGroupCohomology} at various points throughout the proof above. It turns out that the full strength of \cite{BrownGroupCohomology} is not necessary, and one can weaken the finiteness condition $\FP(k)$. Indeed, \cref{thm:main} holds under the assumption that the trivial $k[G]$-module $k$ admits a projective resolution of length $n$ that is finitely generated in degree $n$. This is explained in \cite[Lemma 6.1]{FisherKlinge_RPVN}.
\end{rem}

\begin{proof}[Proof of \cref{cor:free_by_poly_Z}]
    If $G$ is free-by-\{poly-$\Z$\}, then $\b{2}(G) = 0$ by Gaboriau's theorem \cite[Th\'eor\`eme 6.6]{Gaboriau2002}. If $\b{2}(G) = 0$, then $G$ is of type $\FP(\Q)$ by \cite[Theorem 3.10]{JaikinLinton_coherence}, and therefore $G$ admits a poly-$\Z$ quotient with kernel of rational cohomological dimension at most one by \cref{thm:main}. The kernel is free by the Stallings--Swan theorem \cite{Stallings_cd1,Swan_cd1}. \qedhere
\end{proof}

We mention one application to coherence. A group is \emph{coherent} if all its finitely generated subgroups are finitely presented, and a ring is \emph{left coherent} if all its finitely generated left ideals are finitely presented. Note that group algebras are left coherent if and only if they are right coherent, where right coherence is defined analogously.

\begin{cor}\label{cor:coherence}
    Let $G$ be a residually poly-$\Z$ group with $\cd_\Q(G) \leqslant 2$. If $\b{2}(G) = 0$, then $G$ is coherent and $k[G]$ is left coherent for any field $k$.
\end{cor}
\begin{proof}
    The assumptions imply that $\b{2}(H) = 0$ for all subgroups of $G$ \cite[Proposition 3.7]{JaikinLinton_coherence}, and therefore every finitely generated subgroup of $G$ is free-by-\{poly-$\Z$\} by \cref{cor:free_by_poly_Z}. The conclusions of the corollary now essentially follow from the coherence results of \cite{JaikinLinton_coherence}; this is carefully explained in the proof of \cite[Corollary 6.9]{FisherKlinge_RPVN}. \qedhere
\end{proof}

\bibliography{bib}
\bibliographystyle{alpha}

\end{document}